\documentclass[12pt]{amsart}
\usepackage{amsfonts}
 \newtheorem{theorem}{Theorem}[section]
 \newtheorem{cor}[theorem]{Corollary}
 \newtheorem{conjecture}[theorem]{Conjecture}
 \newtheorem{lemma}[theorem]{Lemma}
 
 \theoremstyle{definition}
 
 \theoremstyle{remark}

\numberwithin{equation}{section}
\DeclareMathOperator{\supp}{supp}
\DeclareMathOperator{\ord}{ord}

\begin{document}
\title[On Subset Sums of Zero-sum Free set]
{On the structure of zero-sum free set with minimum subset sums in abelian groups}

\author[J.T. Peng]{Jiangtao Peng$^*$ }
\address{College of Science, Civil Aviation University of China,
Tianjin 300300, P.R. China}\email{jtpeng1982@aliyun.com}

\author[W.Z. Hui]{Wanzhen Hui}
\address{College of Science, Civil Aviation University of China,
Tianjin 300300, P.R. China} \email{huiwanzhen@163.com}



\begin{abstract}
Let $G$ be an additive abelian group and $S\subset G$ a subset. Let $\Sigma(S)$ denote the set of group elements which can be expressed as a sum of a nonempty subset of $S$. We say $S$ is zero-sum free if $0 \not\in \Sigma(S)$. It was conjectured by R.B.~Eggleton and P.~Erd\"{o}s in 1972 and proved by W.~Gao et. al. in 2008 that $|\Sigma(S)|\geq 19$ provided that $S$ is a zero-sum free subset of an abelian group $G$ with $|S|=6$. In this paper, we determined the structure of zero-sum free set $S$ where $|S|=6$ and $|\Sigma(S)|=19$.
\end{abstract}

\date{}

\maketitle

\section{Introduction and main results}
Our notation and terminology are consistent with \cite{GG2006} and \cite{Ge09a}. Let $\mathbb{N}$ and $\mathbb{R}$ be the set of positive integers and real numbers, and $\mathbb{N}_{0}=\mathbb{N}\cup\{0\}$.
For $a,b\in\mathbb{R}$ we set $[a,b]=\{x\in\mathbb{Z}|a\leq x\leq b\}$.

Let $G$ be an additive finite abelian group. Let $\ord (g)$ denote the order of $g \in G$.  Every sequence $S$ over $G$ can be written in the form
\begin{center}
$S=g_1 \cdot  \ldots \cdot g_\ell=\prod_{g \in G} g^{\mathsf v_g (S)}$,
\end{center}
where $\mathsf v_g(S) \in \mathbb{N}_0$ denote the {\it multiplicity} of $g$ in $S$. We call
\begin{itemize}
\item[] $\supp(S) = \{ g \in G \mid \mathsf v_g(S) > 0\}$ \quad  the {\it support } of $S$;
\item[] $\mathsf h(S)=\max \{ \mathsf v_g(S) \mid  g\in G\}$ \quad the {\it maximum of the  multiplicities} of $g$ in $S$;
\item[] $|S|= \ell = \sum_{g\in G}\mathsf v_g(S) \in \mathbb{N}_0$ \quad the {\it length} of $S$;
\item[] $\sigma(S)=\sum_{i=1}^{\ell}g_i = \sum_{g\in G}\mathsf v_g(S)g\in G $ \quad the {\it sum } of $S$.
\end{itemize}

A sequence $T$ is called a {\it subsequence} of $S$ and denoted by $T \mid S$ if $\mathsf v_g (T) \le \mathsf v_g (S)$ for all $g \in G$. Whenever $T \mid S$, let $ST^{-1}$ denote the subsequence with $T$ deleted from $S$. If $S_1, S_2$ are two disjoint subsequences of $S$, let $S_1S_2$ denote the subsequence of $S$ satisfying that $\mathsf v_g(S_1S_2)= \mathsf v_g(S_1) + \mathsf v_g(S_2)$ for all $g \in G$. Let
\begin{center}
$\Sigma(S)=\{ \sigma(T) \mid T \mbox{ is a subsequence of } S  \mbox{ with } 1 \le |T| \le |S|\}.$
\end{center}

The sequence $S$ is called {\it zero-sum }  if $\sigma(S) = 0 \in G$ and {\it zero-sum free}  if $0 \not\in \Sigma(S)$.
If $\sigma(S)=0$ and $\sigma(T)\ne 0$ for every  $T \mid S$  with $1 \le |T| < |S|$, then $S$ is called {\it minimal zero-sum}. If $\mathsf v_g (S) \le 1$ for all $g \in G$, we call $S$ a {\it subset} of $G$.

To characterize the structure of zero-sum free sequences in abelian groups is a subject of great interest in Zero-sum Theory (see \cite{GG1998, GH2002, SC2007, Yuan2007, ZYL2015}). In order to solve these kinds of problems, it becomes a key technology to determine the structure of zero-sum free subsets.

Let $S$ be a zero-sum free subset of $G$. The set $\Sigma(S)$ was first studied by R.B.~Eggleton and P.~Erd\"{o}s in 1972 \cite{EE1972}.  They showed that $|\Sigma(S)|\ge 2|S|-1$ for any subset $S$,  $|\Sigma(S)|\ge 2|S|$ if $|S| \ge 4$, and $|\Sigma(S)|\ge 13$ if $|S|=5$. Also They obtained a zero-sum free subset $S$ of a cyclic group $G$ such that $|\Sigma(S)| = \lfloor \frac{1}{2} |S|^2 \rfloor +1$. They conjectured that
\begin{conjecture}\label{conjecture1}
Let $G$ be an abelian group and $S$ be a zero-sum free subset of $G$. Then $|\Sigma(S)| \ge \lfloor \frac{1}{2} |S|^2 \rfloor +1$. Furthermore the lower bound of $|\Sigma(S)| $ can be achieved by a zero-sum free subset $S$ of a cyclic group.
\end{conjecture}
In 1975, J.E.~Olson \cite{O} proved that $|\Sigma(S)| \ge \frac{1}{9} |S|^2.$ In 2008, W.~Gao et.~al. gave a positive answer to Conjecture \ref{conjecture1} for $|S|=6$.
\begin{theorem}\cite{GLPS}\label{S=6}
Let $G$ be an abelian group and $S$ be a zero-sum free subset of $G$ of length $|S|=6$. Then $|\Sigma(S)| \ge 19$.
\end{theorem}
In 2009, G.~Bhowmik et.~al.\cite{BHS} gave an counterexample of Conjecture~\ref{conjecture1}. They showed that there is a zero-sum free subset $S$ of a cyclic group $G$ of length $|S|=7$ such that $|\Sigma(S)|=24$ (published in 2011). Later, P.~Yuan and X.~Zeng\cite{YZ2010} proved that  for any zero-sum free subset $S$ of an abelian group $G$, $|\Sigma(S)|\ge 24$ if $|S| =7$. In 2010, H.~Guan et.~al.\cite{GZY} described all the zero-sum free subsets $S$ of an abelian $G$ when $S=5$ and $|\Sigma(S)|=13$.

We focus on the following questions:
\begin{itemize}
\item[(1)] What is the lower bound of $|\Sigma(S)|$ for a zero-sum free subset $S$ of an abelian group $G$?
\item[(2)] What is the structure of $S$ when $|\Sigma(S)|$  achieved the  lower bound?
\item[(3)] Can the lower bound of $|\Sigma(S)|$  be achieved by a zero-sum free subset $S$ of a cyclic group?
\end{itemize}

The main aim of the present paper is to determine the structure of zero-sum free subset $S$ of length $|S|=6$ when $|\Sigma(S)|=19$. We now state our main results.

\begin{theorem}\label{main result}
Let $G$ be an abelian group and $S$ be a zero-sum free subset of $G$ of length $|S|=6$. Then $|\Sigma(S)|=19$ if and only if  there exists $x_1, x_2, x_3\in G$ such that $S$ is one of the following forms:
\begin{itemize}
\item[(i)] $S=x_1\cdot x_2 \cdot x_3 \cdot (x_1+x_3)\cdot (x_2+x_3)\cdot (x_1+x_2+x_3)$, where $\ord(x_1)=2$ and $2x_2\in \langle x_1 \rangle$;

\item[(ii)] $S=x_1\cdot x_2\cdot 2x_2\cdot 3x_2\cdot (x_1+x_2)\cdot (x_1+2x_2)$, where $\ord (x_1)=2$;

\item[(iii)] $S=(-2x_1)\cdot x_1 \cdot (3x_1)\cdot (4x_1)\cdot (5x_1)\cdot (6x_1)$, where $\ord(x_1)=20$;

\item[(iv)] $S=(-3x_1)\cdot x_1\cdot (4x_1)\cdot (5x_1)\cdot (9x_1)\cdot (12x_1)$, where $\ord(x_1)=20$;

\item[(v)] $S=x_1\cdot x_2\cdot (x_1+x_2)\cdot (x_1+2x_2)\cdot (2x_1+x_2)\cdot (4x_1+4x_2)$, where  $2x_1=2x_2, \ord(x_1)=\ord(x_2)=10$.
\end{itemize}
\end{theorem}

\begin{cor}\label{corollary}
Let $G$ be an finite abelian group of odd order. Let $S$ be a zero-sum free subset of $G$ of length $|S|=6$. Then $|\Sigma(S)|\geq 20$.
\end{cor}

The paper is organized as follows. In the next section, we provide some preliminary results. In Section 3, we prove  our main results. In the last section, we will give some further remarks.

Throughout this paper, we assume that $G$ is an abelian group.

\section{Preliminaries}

Let $S=g_1 \cdot  \ldots \cdot g_\ell$ be a sequence over $G$. Given any group homomorphism $\varphi: G\rightarrow G'$, where $G'$ is an abelian group,  then $\varphi(S)=\varphi(g_1)\cdot\ldots\cdot\varphi(g_l)$ is a sequence over $G'$.

\begin{lemma}\cite{GH2006}\label{1,t}
Let $S=S_1 \cdot  \ldots \cdot S_t$ be a zero-sum free sequence over $G$, where $S_1$, \ldots, $S_t$ are disjoint subsequences of $S$. Then
\begin{equation*}
|\Sigma(S)|\geq |\Sigma(S_1)| + \ldots + |\Sigma(S_t)|.
\end{equation*}
\end{lemma}

\begin{lemma}\cite{EE1972,GH2006}\label{smallset}
Let $S$ be a zero-sum free subset of $G$. Then
\begin{itemize}
\item[(1)] $|\Sigma(S)|= 2|S|-1$ if $|S|=1$ or $2$;
\item[(2)] $|\Sigma(S)|\ge 5$ if $|S|=3$, furthermore  $|\Sigma(S)|\ge 6$ if $S$ contains no elements of order 2;
\item[(3)] $|\Sigma(S)|\ge 2|S|$ if $|S| \ge 4$;
\item[(4)] $|\Sigma(S)|\ge 13$ if $|S|=5$.
\end{itemize}
\end{lemma}

\begin{lemma}\cite[Theorem 3.2]{GLPS}\label{lemma 2}
Let $S$ be a zero-sum free subset of $G$ of length $|S|\in [4,7]$. If $S$ contains
some elements of order 2, then
\begin{equation*}
|\Sigma(S)|\geq \lfloor \frac{1}{2}|S|^2\rfloor +1.
\end{equation*}
\end{lemma}

\begin{lemma}\cite[Lemma 2.3]{GLPS}\label{phis2}
Let $S=S_1S_2$ be a zero-sum free sequence over $G$, where $S_1$, $S_2$ are disjoint subsequences of $S$. Let $H=\langle \supp(S_1)\rangle$ and let $\varphi: G\rightarrow G/H$ denote the canonical epimorphism. Then we have
\begin{equation*}
|\Sigma(S)|\geq (1+|\Sigma(\varphi(S_2))|)|\Sigma(S_1)|+|\Sigma(\varphi(S_2))|.
\end{equation*}
\end{lemma}

\begin{lemma}\cite[Lemma 2.4]{GLPS}\label{lemma T}
\begin{itemize}
\item[(1)] Suppose $U, T, S$ are three sequences over $G$ and $S$ is zero-sum free. If $U\mid T$ and $TU^{-1}\mid S$, then $\sigma(U)\neq \sigma(T)$.

\item[(2)] Suppose $T_1,T_2$ are two subsets of $G$ with $|T_1|=|T_2|$ and $|T_1\cap T_2|=|T_1|-1$, then $\sigma(T_1)\neq \sigma(T_2)$.
\end{itemize}
\end{lemma}

Let $S=x_1\cdot \ldots \cdot x_k$ be a zero-sum free subset of $G$ of length $|S|=k\in \mathbb{N}$, and let $\mathcal{A}$ be the set of all nonempty subsets of $S$. We partition $\mathcal{A}$
as
\begin{equation*}
\mathcal{A}=\mathcal{A}_1\uplus \ldots \uplus \mathcal{A}_r,
\end{equation*}
where two subsets $T,T'$ of $S$ are in the same class $\mathcal{A}_{\nu}$, for some $\nu \in [1,r]$, if $\sigma(T)=\sigma(T')$. Thus we have $r=|\Sigma(S)|$. For a subset $\mathcal{B}\subset \mathcal{A}$
we set
\begin{equation*}
\overline{\mathcal{B}}=\{ST^{-1}|T\in \mathcal{B}\}.
\end{equation*}
Then, for every $\nu \in [1,r]$, we clearly have $\overline{\mathcal{A}_{\nu}}\in \{\mathcal{A}_1,...,\mathcal{A}_r\}$, and $\overline{\mathcal{A}_{\nu}}$ will be called the dual class of $\mathcal{A}_{\nu}$. For a nontrivial subset $T$ of $S$ we denote by $[T]$ the class of $T$.

\begin{lemma}\cite[Lemma 3.3]{GLPS}\label{4,5}
Let $S=x_1\cdot \ldots \cdot x_6$ be a zero-sum free subset of $G$ with $\ord(x_i) \ge 3$ for $i=1,\ldots,6$. $[x_i]$ and $\mathcal{A}_j$ are defined as above.  Then $|[x_i]|\leq 4$ for every $i\in [1,6]$, and $|\mathcal{A}_i|\leq 5$ for every $i\in [1,r]$.
\end{lemma}

Let $P_n$ denote the symmetric group on $[1,n]$.

\begin{lemma}\cite[Lemma 7.2]{GLPS}\label{x4}
Let $S=x_1\cdot \ldots \cdot x_6$ be a zero-sum free subset of $G$ with $\ord(x_i) \ge 3$ for $i=1,\ldots,6$. $[x_i]$ is defined as above. If $|[x_i]|=4$ for some $i=[1,6]$, then there exists $\tau\in P_6$ such that $[x_i]$ is one of the following forms:
\begin{itemize}
\item[(b1)] $\{x_{\tau(1)}, x_{\tau(2)}\cdot x_{\tau(3)}\cdot x_{\tau(4)}\cdot x_{\tau(5)}, x_{\tau(2)}\cdot x_{\tau(6)},
x_{\tau(3)}\cdot x_{\tau(4)}\cdot x_{\tau(6)}\}$;

\item[(b2)] $\{x_{\tau(1)}, x_{\tau(2)}\cdot x_{\tau(3)}\cdot x_{\tau(4)}\cdot x_{\tau(5)}, x_{\tau(2)}\cdot x_{\tau(3)}\cdot x_{\tau(6)}, x_{\tau(4)}\cdot x_{\tau(5)}\cdot x_{\tau(6)}\}$;

\item[(b3)] $\{x_{\tau(1)}, x_{\tau(2)}\cdot x_{\tau(3)}, x_{\tau(4)}\cdot x_{\tau(5)}, x_{\tau(2)}\cdot x_{\tau(4)}\cdot x_{\tau(6)}\}$;

\item[(b4)] $\{x_{\tau(1)}, x_{\tau(2)}\cdot x_{\tau(3)}\cdot x_{\tau(4)}, x_{\tau(2)}\cdot x_{\tau(5)}\cdot x_{\tau(6)}, x_{\tau(3)}\cdot x_{\tau(5)}\}$.
\end{itemize}
\end{lemma}

\begin{lemma}\cite[Lemma 7.9]{GLPS}\label{x5}
Let $S=x_1\cdot \ldots \cdot x_6$ be a zero-sum free subset of $G$ with $\ord(x_i) \ge 3$ for $i=1,\ldots,6$. $[x_i]$ and $\mathcal{A}_j$ are defined as above. If $|\mathcal{A}_i|=5$, then there exists $\tau\in P_6$ such that $\mathcal{A}_i$ is one of the following forms:
\begin{itemize}
\item[(c1)] $\{x_{\tau(1)}\cdot x_{\tau(2)}, x_{\tau(3)}\cdot x_{\tau(4)}, x_{\tau(1)}\cdot x_{\tau(3)}\cdot x_{\tau(5)}\cdot x_{\tau(6)}, x_{\tau(2)}\cdot x_{\tau(4)}\cdot x_{\tau(5)}\cdot x_{\tau(6)},
x_{\tau(1)}\cdot x_{\tau(4)}\cdot x_{\tau(5)}\}$;

\item[(c2)] $\{x_{\tau(1)}\cdot x_{\tau(2)}, x_{\tau(3)}\cdot x_{\tau(4)}, x_{\tau(1)}\cdot x_{\tau(3)}\cdot x_{\tau(5)}\cdot x_{\tau(6)}, x_{\tau(2)}\cdot x_{\tau(5)}\cdot x_{\tau(6)},
x_{\tau(1)}\cdot x_{\tau(4)}\cdot x_{\tau(5)}\}$;

\item[(c3)] $\{x_{\tau(1)}\cdot x_{\tau(2)}, x_{\tau(3)}\cdot x_{\tau(4)}, x_{\tau(1)}\cdot x_{\tau(3)}\cdot x_{\tau(5)}\cdot x_{\tau(6)}, x_{\tau(1)}\cdot x_{\tau(4)}\cdot x_{\tau(5)},
x_{\tau(2)}\cdot x_{\tau(4)}\cdot x_{\tau(6)}\}$;

\item[(c4)] $\{x_{\tau(1)}\cdot x_{\tau(2)}, x_{\tau(3)}\cdot x_{\tau(4)}, x_{\tau(1)}\cdot x_{\tau(5)}\cdot x_{\tau(6)}, x_{\tau(2)}\cdot x_{\tau(3)}\cdot x_{\tau(5)},
x_{\tau(2)}\cdot x_{\tau(4)}\cdot x_{\tau(6)}\}$;

\item[(c5)] $\{x_{\tau(1)}\cdot x_{\tau(2)}, x_{\tau(1)}\cdot x_{\tau(3)}\cdot x_{\tau(5)}\cdot x_{\tau(6)}, x_{\tau(1)}\cdot x_{\tau(3)}\cdot x_{\tau(4)}, x_{\tau(2)}\cdot x_{\tau(3)}\cdot x_{\tau(6)},
x_{\tau(4)}\cdot x_{\tau(5)}\cdot x_{\tau(6)}\}$;

\item[(c6)] $\{x_{\tau(1)}\cdot x_{\tau(2)}, x_{\tau(1)}\cdot x_{\tau(3)}\cdot x_{\tau(5)}\cdot x_{\tau(6)}, x_{\tau(1)}\cdot x_{\tau(3)}\cdot x_{\tau(4)}, x_{\tau(2)}\cdot x_{\tau(3)}\cdot x_{\tau(6)},
x_{\tau(2)}\cdot x_{\tau(4)}\cdot x_{\tau(5)}\}$;

\item[(c7)] $\{x_{\tau(1)}\cdot x_{\tau(2)}, x_{\tau(1)}\cdot x_{\tau(3)}\cdot x_{\tau(4)}, x_{\tau(1)}\cdot x_{\tau(5)}\cdot x_{\tau(6)}, x_{\tau(2)}\cdot x_{\tau(3)}\cdot x_{\tau(5)},
x_{\tau(2)}\cdot x_{\tau(4)}\cdot x_{\tau(6)}\}$.
\end{itemize}
\end{lemma}

The following lemma gives the structure of a zero-sum free subset $T$ of $G$, where $|T|=4$ and $|\Sigma(T)|$  achieves the  lower bound.
\begin{lemma}\label{4}
Let $T$ be a zero-sum free subset of $G$ of length $|T|=4$. Then $|\Sigma(T)|=8$ if and only if there exists $x \in G$ such that $T=x\cdot(3x)\cdot(4x)\cdot(7x)$ and $\ord(x)=9$.
\end{lemma}

\begin{proof}
Clearly, if $T=x\cdot(3x)\cdot(4x)\cdot(7x)$ and $\ord(x)=9$, we have $|\Sigma(T)|=8$.
Next, assume that $T=x_1\cdot x_2\cdot x_3\cdot x_4$ and $|\Sigma(T)|=8$, by Lemma~\ref{lemma 2}, we have $T$ contains no elements of order 2.

{\bf Claim}: There exists $\{i_1,i_2,i_3\}\subset [1,4]$, such that $|\Sigma(x_{i_1}\cdot x_{i_2}\cdot x_{i_3})|=7$.

Assume to the contrary that $|\Sigma(x_{i_1}\cdot x_{i_2}\cdot x_{i_3})|\le 6$ for every $\{i_1,i_2,i_3\}\subset [1,4]$. Since $T$ contains no elements of order 2, by Lemma \ref{smallset}   we have $|\Sigma(x_{i_1}\cdot x_{i_2}\cdot x_{i_3})|= 6$. Note that $|\Sigma(x_1\cdot x_2\cdot x_3)|=6$ implies that $x_1=x_2+x_3$ or $x_2=x_1+x_3$ or $x_3=x_1+x_2$. Without loss of generality assume that  $x_1=x_2+x_3$. Since $|\Sigma(x_1\cdot x_2\cdot x_4)|=6$, we infer that $x_2=x_1+x_4$ or $x_4=x_1+x_2$. If $x_2=x_1+x_4$, then $x_1+x_2=x_2+x_3+x_1+x_4$ implies that $x_3+x_4=0$, giving a contradiction. Hence $x_4=x_1+x_2$. Since $|\Sigma(x_1\cdot x_3\cdot x_4)|=6$, we have $x_3=x_1+x_4$. Then $x_1+x_3=x_2+x_3+x_1+x_4$, which implies that $x_2+x_4=0$, giving a contradiction. This proves the Claim.

By the Claim we may assume that $|\Sigma(x_1\cdot x_2\cdot x_3)|=7$. Then
\begin{equation*}
M=\{x_4\} \cup (x_4 +\Sigma(x_1\cdot x_2\cdot x_3))
\end{equation*}
is a subset of $\Sigma(T)$ with $|M|=8$. Therefore $\Sigma(x_1\cdot x_2\cdot x_3) \subset M= \Sigma(T)$. Since $T$ is a zero-sum free subset, we infer that $ x_1\in \{x_2+x_4, x_3+x_4, x_2+x_3+x_4\}.$ We distinguish three cases.

{\bf Case 1.} $x_1=x_2+x_4$. Since $T$ is a zero-sum free subset, we have that $x_2\in \{x_3+x_4, x_1+x_3+x_4\}$ and $x_3\in \{x_1+x_4, x_1+x_2+x_4\}$.

First assume that $x_2=x_3+x_4$. If $x_3=x_1+x_4$, again since $T$ is a zero-sum free subset, we infer that $x_1+x_2+x_3= x_4$. Then $x_2=7x_1, x_3=4x_1, x_4=3x_1$ and $\ord(x_1)=9$. Let $x=x_1$ and we are done.  If $x_3=x_1+x_2+x_4$, then $x_2+x_3\notin M$, yielding a contradiction.

Next assume that $x_2=x_1+x_3+x_4$. Then $x_1+x_2\notin M$, yielding a contradiction too.

{\bf Case 2.} $x_1=x_3+x_4$. Similar to Case 1.

{\bf Case 3.} $x_1=x_2+x_3+x_4$. We may also assume that $x_2=x_1+x_3+x_4$ and $x_3=x_1+x_2+x_4$, or it reduce to Case 1 or Case 2. But these imply that $x_1+x_2=x_4$ and $x_1+x_3=x_4$, which is impossible.
\end{proof}

\section{Proof of Theorem \ref{main result}}
In this section, let
\begin{equation*}
S=x_1\cdot \ldots \cdot x_6
\end{equation*}
be a zero-sum free subset of $G$ and let $\mathcal{A}_1,\ldots,\mathcal{A}_r$
be as introduced in the second section. Let (bi), (cj) be defined as in Lemma~\ref{x4}, Lemma~\ref{x5}.

\begin{lemma}\label{lemma for order 2}
Suppose $\ord(x_1)=2$. Then $|\Sigma(S)|=19$ if and only if  there exists $g, h\in G$ such that $S$ is one of the following forms:
\begin{itemize}
\item[(a1)] $S=x_1\cdot g \cdot h\cdot (x_1+h)\cdot (g+h)\cdot (x_1+g+h)$, where $2g \in \langle x_1 \rangle$;

\item[(a2)] $S=x_1\cdot g\cdot 2g\cdot 3g\cdot (x_1+g)\cdot (x_1+2g)$.
\end{itemize}
\end{lemma}
\begin{proof}
Since $S$ is zero-sum free, it can be easily proved that $|\Sigma(S)|=19$ when $S$ is of form (a1) or (a2). Next assume that $|\Sigma(S)|=19$. We set
\begin{center}
$S=S_1S_2$ where $S_1=x_1$ and $S_2=x_2\cdot \ldots \cdot x_6$.
\end{center}
Let $H=\langle x_1\rangle=\{0,x_1\}$ and $\varphi: G\rightarrow G/H$ the canonical epimorphism, then $\varphi(S_2)=\varphi(x_2)\cdot \ldots \cdot \varphi(x_6)$. It follows from the proof of Theorem~3.2 in \cite{GLPS} that
\begin{center}
$\varphi(S_2)$ is zero-sum free and $\mathsf h(\varphi(S_2))\leq 2$.
\end{center}
So we have $|\supp(\varphi(S_2))|\geq3$. Since $|\Sigma(S)|=19$, by Lemma~\ref{phis2}, we infer that $$|\Sigma(\varphi(S_2))|\leq 9.$$ It follows from Lemma~\ref{smallset} that $|\supp(\varphi(S_2))|\leq 4$.

We first show that $|\supp(\varphi(S_2))|\neq 4$. Otherwise we can write $\varphi(S_2)=a^2\cdot b\cdot c \cdot d$, where $a, b, c, d \in G/H$ are pairwise distinct. Let $U_1=a \cdot b\cdot c \cdot d$ and $U_2=a$. If $|\Sigma(U_1)|\ge 9$, then by Lemma~\ref{1,t} we have $|\Sigma(\varphi(S_2))| \ge |\Sigma(U_1)|+|\Sigma(U_2)| \ge 10$, yielding a contradiction. Hence $|\Sigma(U_1)| \le 8.$ By Lemma~\ref{4}, there exists $x\in G/H$ such that $U_1=x\cdot(3x)\cdot(4x)\cdot(7x)$ and $\ord(x)=9$. Hence  $\varphi(S_2)=(ix) \cdot x\cdot(3x)\cdot(4x)\cdot(7x)$, where $i \in \{1,3,4,7\}$. This is impossible
since $\varphi(S_2)$ is zero-sum free.

Next assume that $|\supp(\varphi(S_2))|=3$. Since $|S_2|=5$ we may assume that $\varphi(S_2)=a^2\cdot b^2\cdot c$. Since $\varphi(S_2)$ is zero-sum free, we must have $\ord(a)\neq 2$ and $\ord(b)\neq 2$. Let $U_1=a\cdot b\cdot c, U_2=a\cdot b$. By Lemma~\ref{smallset}, we have $|\Sigma(U_2)|=3$. Note that $|\Sigma(\varphi(S_2))|\leq 9$, it follows from Lemma~\ref{1,t} that $|\Sigma(U_1)|\leq 6$. Then we have $a=b+c, b=a+c$ or $c=a+b$. We distinguish three cases.

{\bf Case 1.} $a=b+c$. We first show that $\ord(c)=2$. If  $\ord(c)\neq 2$, let
$A=\{a, b, a+b, 2a, 2b, 2a+b, a+2b, 2a+2b, 2a+2b+c\} \subset \Sigma(\varphi(S_2)).$
Then $A$ is a set of 9 distinct elements and hence $A= \Sigma(\varphi(S_2))$. An easy calculation shows that $|\{ 2a+b+c, a+c\} \cap A| \le 1$, yielding a contradiction. Hence $\ord(c)=2$. Assume that $\varphi(x_2)=c, \varphi(x_3)=\varphi(x_4)=a,  \varphi(x_5)=\varphi(x_6)=b$. Let $g=x_2, h=x_5$, then $2g \in \langle x_1 \rangle$ and we can write $S$ as $S=x_1\cdot g \cdot h\cdot (x_1+h)\cdot (g+h)\cdot (x_1+g+h)$.

{\bf Case 2.} $b=a+c$. Similar to Case 1.

{\bf Case 3.} $c=a+b$. Since $\varphi(S_2)$ is zero-sum free, we have $\ord(c)\neq 2$. Let
$B=\{a, b, a+b, 2a+b, a+2b, 2a+2b, 2a+b+c, a+2b+c, 2a+2b+c\}\subset \Sigma(\varphi(S_2)).$ Then $B$ is a set of 9 distinct elements and hence $B= \Sigma(\varphi(S_2))$. Thus $2a \in B$. An easy calculation shows that $2a=b$ or $2a=a+2b$. In both cases we can write $S$ as $S=x_1\cdot g\cdot 2g\cdot 3g\cdot (x_1+g)\cdot (x_1+2g)$.
\end{proof}

\begin{lemma}\label{20}
Suppose $S$ contains no elements of order $2$ and there exists $j \in [1,r]$ and $\tau\in P_6$ such that $\mathcal{A}_j \in  \{\mbox{b2, c7}\}.$ Then $|\Sigma(S)|\geq 20$.
\end{lemma}
\begin{proof}
We just prove the case that $\mathcal{A}_j$ is of form (b2), the other case is similar and we will give the proof in the APPENDIX. Without loss of generality we assume that $x_1=x_2+x_3+x_4+x_5=x_2+x_3+x_6=x_4+x_5+x_6$. Assume to the contrary that $|\Sigma(S)| \leq 19$. Let

$a_1=x_1=x_2+x_3+x_4+x_5=x_2+x_3+x_6=x_4+x_5+x_6$,

$a_2=x_2$,

$a_3=x_4$,

$a_4=x_6=x_2+x_3=x_4+x_5$,

$a_5=x_1+x_2=x_2+x_4+x_5+x_6$,

$a_6=x_1+x_4=x_2+x_3+x_4+x_6$,

$a_7=x_1+x_6=x_1+x_2+x_3=x_1+x_4+x_5=x_2+x_3+x_4+x_5+x_6$,

$a_8=x_2+x_4$,

$a_9=x_2+x_6=x_2+x_4+x_5$,

$a_{10}=x_4+x_6=x_2+x_3+x_4$,

$a_{11}=x_1+x_2+x_4$,

$a_{12}=x_1+x_2+x_6=x_1+x_2+x_4+x_5$,

$a_{13}=x_1+x_4+x_6=x_1+x_2+x_3+x_4$,

$a_{14}=x_2+x_4+x_6$,

$a_{15}=x_1+x_2+x_3+x_6=x_1+x_4+x_5+x_6=x_1+x_2+x_3+x_4+x_5$,

$a_{16}=x_1+x_2+x_4+x_6$,

$a_{17}=x_1+x_2+x_3+x_4+x_6$,

$a_{18}=x_1+x_2+x_4+x_5+x_6$,

$a_{19}=x_1+x_2+x_3+x_4+x_5+x_6$.

A straightforward computation shows that $a_1, a_2, \ldots , a_{19}$ are pairwise distinct. Then
\begin{equation*}
A=\{a_i \mid 1\le i \le 19 \} = \Sigma(S),
\end{equation*}
and hence $a_{20}=x_1+x_3=x_3+x_4+x_5+x_6 \in \Sigma(S)$. Since $S$ is zero-sum free and $S$ contains no elements of order 2, we infer that $a_{20}\in \{a_2,a_8,a_{11},a_{14},a_{16},a_{18}\}$.

If $a_{20}\in \{a_2,a_8,a_{11}\}$, it is easy to verify that $a_{21}=x_3+x_6=x_3+x_4+x_5 \not\in A$, yielding a contradiction. If $a_{20}\in \{a_{14},a_{16},a_{18}\}$, then $a_{22}=x_1+x_3+x_6=x_1+x_3+x_4+x_5 \not\in A$, yielding a contradiction too.
\end{proof}

\begin{lemma}\label{19}
Suppose $S$ contains no elements of order $2$ and there exists $j \in [1,r]$ and $\tau\in P_6$ such that $\mathcal{A}_j \in  \{\mbox{b1, b3, b4, c1, c2, ..., c6}\}.$ Then $|\Sigma(S)|=19$ if and only if there exists $g, h\in G$ such that $S$ is one of the following forms:
\begin{itemize}
\item[(a3)]$S=(-2g)\cdot g\cdot (3g)\cdot (4g)\cdot (5g)\cdot (6g)$ and $\ord(g)=20$;

\item[(a4)] $S=(-3g)\cdot g\cdot (4g)\cdot (5g)\cdot (9g)\cdot (12g)$ and $\ord(g)=20$;

\item[(a5)] $S=(g)\cdot h\cdot (g+h)\cdot (g+2h)\cdot (2g+h)\cdot (4g+4h)$ and $2g=2h$, $\ord(g)=\ord(h)=10$.
\end{itemize}
\end{lemma}
\begin{proof} By Theorem~\ref{S=6}, $|\Sigma(S)|\ge 19$ in all cases. Next assume that $|\Sigma(S)|= 19$.  We just prove the case that $\mathcal{A}_j$ is of form (b3), the other cases are similar and we will give the proofs in the APPENDIX. Without loss of generality we assume that $x_1=x_2+x_3=x_4+x_5=x_2+x_4+x_6$. Let

$a_1=x_1=x_2+x_3=x_4+x_5=x_2+x_4+x_6$,

$a_2=x_2$,

$a_3=x_3=x_4+x_6$,

$a_4=x_4$,

$a_5=x_5=x_2+x_6$,

$a_{6}=x_1+x_2=x_2+x_4+x_5$,

$a_7=x_1+x_3=x_1+x_4+x_6=x_3+x_4+x_5=x_2+x_3+x_4+x_6$,

$a_{8}=x_1+x_4=x_2+x_3+x_4$,

$a_9=x_1+x_5=x_1+x_2+x_6=x_2+x_3+x_5=x_2+x_4+x_5+x_6$,

$a_{10}=x_1+x_6=x_3+x_5=x_2+x_3+x_6=x_4+x_5+x_6$,

$a_{11}=x_2+x_4$,

$a_{12}=x_1+x_2+x_3=x_1+x_4+x_5=x_1+x_2+x_4+x_6=x_2+x_3+x_4+x_5$,

$a_{13}=x_1+x_3+x_5=x_1+x_2+x_3+x_6=x_1+x_4+x_5+x_6=x_2+x_3+x_4+x_5+x_6$,

$a_{14}=x_1+x_2+x_3+x_5=x_1+x_2+x_4+x_5+x_6$,

$a_{15}=x_1+x_3+x_4+x_5=x_1+x_2+x_3+x_4+x_6$,

$a_{16}=x_1+x_2+x_3+x_4+x_5$,

$a_{17}=x_1+x_2+x_3+x_4+x_5+x_6$.

Since $S$ is a zero-sum free subset and $S$ contains no elements of order~2, by Lemma~\ref{lemma T} we infer that $a_1,a_2,\ldots,a_{17}$ are pairwise distinct. Let $A=\{a_1, a_2,\ldots, a_{17}\} \subset \Sigma(S).$ Then $|\Sigma(S) \setminus A |=2$. Note that $$a_{18}=x_1+x_2+x_4 \notin A\setminus \{a_3, a_{5}, a_{10}\}.$$ We distinguish four cases.

{\bf Case 1.} $a_{18}=a_3$. That is $x_1+x_2+x_4=x_3=x_4+x_6$, and then $x_6=x_1+x_2=x_2+x_4+x_5$. Since $S$ is a zero-sum free subset and $S$ contains no elements of order 2, we infer that
$$a_{19}=x_1+x_2+x_3+x_4\notin A\setminus \{a_{5}\};$$
$$a_{20}=x_1+x_3+x_4+x_5+x_6\notin A\setminus \{a_2\}.$$

If $a_{19}=a_5$, then $x_1+x_2+x_3+x_4=x_5=x_2+x_6$. Therefore $x_6=x_1+x_2=x_2+x_4+x_5=x_1+x_3+x_4$ and $x_2=x_3+x_4$. These imply that $$x_1=4x_2,\, x_3=3x_2,\, x_4=-2x_2,\, x_5=6x_2, \, x_6=5x_2.$$ Since $|\Sigma(S)|=19$, we infer that $\ord(x_2)=20$. Let $g=x_2$, then $S$ is of form (a3).

If $a_{20}=a_2$, then $x_1+x_3+x_4+x_5+x_6=x_2$. Then $x_2=4x_1,\, x_3=-3x_1,\, x_4=-8x_1,\ x_5=9x_1, \, x_6=5x_1.$ Since $|\Sigma(S)|=19$, we infer that $\ord(x_2)=20$. Let $g=x_1$, then $S$ is of form (a4).

Next we assume that $a_{19} \ne a_5$ and $a_{20} \ne a_2$. Then $a_{19}, a_{20} \not \in A$. Since $S$ contains no elements of order $2$, we have that $a_{19} \ne a_{20}$. Therefore $$\Sigma(S)= A \cup\{a_{19}, a_{20}\}.$$ Now $a_{21}=x_1+x_3+x_4+x_6\in \Sigma(S)= A \cup\{a_{19}, a_{20}\}.$ Note that $a_{19} \ne a_5$ means $x_2\ne x_3+x_4$. Again since $S$ is a zero-sum free subset and $S$ contains no elements of order 2, we have $$a_{21}=x_1+x_3+x_4+x_6\in \{a_2, a_5\}.$$

If $a_{21}=x_1+x_3+x_4+x_6=a_2=x_2$, we infer that $x_1+x_3+x_4\not \in \Sigma(S)$, yielding a contradiction. If $a_{21}=x_1+x_3+x_4+x_6=a_5=x_5=x_2+x_6$, we have that $x_2=x_1+x_3+x_4$. Then $x_1+x_2+x_5\not \in \Sigma(S)$, yielding a contradiction, too.

{\bf Case 2.} $a_{18}=a_5$. Let $\rho \in P_6$ such that $$\rho(1)=1,\, \rho(2)=4,\, \rho(3)=5, \, \rho(4)=2,\, \rho(5)=3, \rho(6)=6.$$ Replace $S$ by $\rho(S)$, it reduce to Case 1.

{\bf Case 3.} $a_{18}=a_{10}$. That is $x_1+x_2+x_4=x_1+x_6=x_3+x_5=x_2+x_3+x_6=x_4+x_5+x_6$, and then $x_6=x_2+x_4$. Since $S$ is a zero-sum free subset and $S$ contains no elements of order 2, we infer that
$$a_{19}=x_1+x_2+x_3+x_4\notin A;$$
$$a_{22}=x_1+x_2+x_4+x_5\notin A.$$

By Lemma~\ref{lemma T}, $a_{19} \ne a_{22}$. Therefore $\Sigma(S)= A \cup\{a_{19}, a_{20}\}.$ Since
$$a_{20}=x_1+x_3+x_4+x_5+x_6\in \Sigma(S),$$ we infer that $a_{20}=a_2$, that is $x_1+x_3+x_4+x_5+x_6=x_2$. This forces that $a_{21}=x_1+x_3+x_4+x_6\notin \Sigma(S),$ which is impossible.

{\bf Case 4.} $a_{18}\not\in \{a_3,\, a_5,\, a_6\}$. Then $a_{18}\notin A$ and $x_6 \not\in \{ x_1+x_2, x_1+x_4, x_2+x_4\}$. Let
\begin{equation*}
M=A\cup\{a_{18}\} \subset \Sigma(S).
\end{equation*}

Since $x_6 \not\in \{ x_1+x_2, x_1+x_4, x_2+x_4\}$, we obtain that $$a_{19}=x_1+x_2+x_3+x_4\notin M\setminus \{a_5\};$$ $$a_{22}=x_1+x_2+x_4+x_5 \notin M\setminus \{a_3\}.$$
If  $a_{19}=a_{5}$ and $a_{22}=a_3$, then $x_6=x_1+x_3+x_4=x_1+x_2+x_5$. Since $S$ is a zero-sum free subset and $S$ contains no elements of order~2, we infer that $x_3+x_4\notin M$ and $a_{21}=x_1+x_3+x_4+x_6\notin M$. Since  $x_3+x_4 \ne a_{21}=x_1+x_3+x_4+x_6$, we obtain that $ |\Sigma(S)| \ge |M|+2 \ge 20$, yielding a contradiction. Therefore either $a_{19} \ne a_{5}$ or  $a_{22}\ne a_3$. Note that $a_{19} \ne a_{22}$ and $|\Sigma(S)|=19$, we have that $|M\cap \{a_{19}, a_{22}\}| = 1.$ Let $\rho \in P_6$ be defined as in Case 2. Replace $S$ by $\rho(S)$ if necessary, we may assume that $a_{19} \notin M$ and $a_{22} \in M$. Then $\Sigma(S)=M\cup \{a_{19}\}$ and
$x_1+x_2+x_4+x_5=x_3=x_4+x_6$. It is easy to verify that $$a_{23}=x_1+x_3+x_6=x_3+x_4+x_5+x_6\notin \Sigma(S),$$ yielding a contradiction. This completes the proof.
\end{proof}

Now we are in a position to prove Theorem~\ref{main result}.

{\it Proof of Theorem~\ref{main result}}. If $S$ is one of the form in $(i)-(v)$, it is easy to verify that $|\Sigma(S)|=19$.

Next we assume that $|\Sigma(S)|=19$. If $S$ contains some elements of order 2, by Lemma~\ref{lemma for order 2}, we have $S$ is of form $(i)$ or $(ii)$. Then we may assume that $S$ contains no elements of order 2. By Lemma \ref{4,5}, $|\mathcal{A}_i|\leq 5$ for all $i\in [1,r]$.

If there exists $i \in [1,r]$ such that $|\mathcal{A}_i|=5$, then by Lemma~\ref{x5} and Lemma~\ref{20}, we have there exists $\tau\in P_6$ such that $\mathcal{A}_j \in  \{\mbox{c1, c2, ..., c6}\}.$ By Lemma~\ref{19},  $S$ is of form $(iii)$, $(iv)$ or $(v)$ and we are done. Hence we may assume that $|\mathcal{A}_i|\leq 4$ for all $i\in [1,r]$

If there exists $i\in [1,r]$ such that $|\mathcal{A}_i|=4$, then by Lemma~\ref{x5} and Lemma~\ref{20}, we have $\mathcal{A}_i$ is not of form (b2). If there exists $\tau\in P_6$ such that $\mathcal{A}_j \in  \{\mbox{b1, b3, b4}\}, $ then by Lemma~\ref{19},  $S$ is of form $(iii)$, $(iv)$, or $(v)$ and we are done. Hence we may assume that $|[x_i]| \le 3$ for $i \in [1,6]$.

Let $t$ denote the number of those $i\in [1,6]$ such that $[x_i]=\overline{[x_i]}$. Without loss of generality assume that $\mathcal{A}_{i}=[x_i]=\overline{[x_i]}$ for $i\in [1,t]$. Since $S$ is a zero-sum free subset of $G$ and noting that $x_i, Sx_i^{-1} \in [x_i]$, we have $[x_i]=\{x_i, Sx_i^{-1}\}$. That is $|\mathcal{A}_{i}| =2$ for $i \in [1,t]$. Also since $S$ is a zero-sum free subset, we infer that if $i \ne j$, then $ \overline{[x_i]} \ne [x_j]$. Without loss of generality assume that $\mathcal{A}_{i}=[x_i]$ for $i\in [t+1,6]$ and $\mathcal{A}_{13-i}=\overline{[x_i]}$ for $i\in [t+1,6]$. Then $|\mathcal{A}_{i}| \le 3$ for $i \in [t+1, 12-t]$.  Now assume that $\mathcal{A}_{r}=\{S\}$, we have
\begin{align*}
63&=\sum_{j=1}^{r}|\mathcal{A}_j|
=\sum_{i=1}^{t}|\mathcal{A}_{i}|+\sum_{i=t+1}^{12-t}|\mathcal{A}_{i}|+\sum_{i=13-t}^{r-1}|\mathcal{A}_{i}|+|\mathcal{A}_r|
\\&\le 2t+3(12-2t)+4(r-1-(12-t))+1= 4r-19.
\end{align*}
Hence $r \ge \frac{82}{4} >20$, yielding a contradiction. This completes the proof. \qed

\section{Concluding Remarks}

In this section, we will give some concluding remarks.

{\it Proof of Corollary \ref{corollary}.} Let $G$ be an finite abelian group of odd order. Let $S$ be a zero-sum free subset of $G$ of length $|S|=6$. Note that $|G|$ is odd, we have that $S$ is not of any form in Theorem~\ref{main result}. Then by Theorem~\ref{S=6} and Theorem~\ref{main result}, we infer that $|\Sigma(S)|\ge 20$.    \qed

\begin{lemma}\label{S=5} \cite{GZY}
Let $G$ be an abelian group and $S$ be a zero-sum free subset of $G$ of length $|S|=5$. Then $|\Sigma(S)|=13$ if and only if  there exists $x_1, x_2\in G$ such that $S$ is one of the following forms:
\begin{itemize}
\item[(i)] $S=x_1\cdot x_2 \cdot (x_1+x_2)\cdot 2x_2\cdot (2x_1+x_2)$, where $\ord(x_1)=2$;

\item[(ii)] $S=(-2x_1)\cdot x_1 \cdot (3x_1)\cdot (4x_1)\cdot (5x_1)$, where $\ord(x_1)=14$;
\end{itemize}
\end{lemma}

Similar to the proof of Corollary~\ref{corollary}, Lemma~\ref{S=5} implies the following result.
\begin{cor}\label{corollary2}
Let $G$ be an finite abelian group of odd order. Let $S$ be a zero-sum free subset of $G$. Then $|\Sigma(S)|\ge 14$ if $|S|=5$.
\end{cor}

Now based on Corollaries~\ref{corollary} and \ref{corollary2}, similar with the proof of Theorem~1.3 in \cite{GLPS}, we can show that
\begin{cor}\label{corollary3}
Let $G$ be a finite cyclic group of odd order. Let $S$ be a zero-sum free sequence over $G$. If $|S| \ge \frac{6|G|+26}{20}$, then $S$ contains an element $g\in G$ with multiplicity
$$
\mathsf v_g(S)\geq \frac{6|S|-|G|+1}{16}.
$$
\end{cor}

\bigskip

\noindent {\bf Acknowledgements}.  This work has been supported by  the National Science Foundation of China (Grant No.~11301531) and the Fundamental Research Funds for the Central Universities (No.~3122017075).

\newpage
{\bf APPENDIX}

\begin{center}
{\bf Proof of Lemma \ref{20}}
\end{center}

We just prove the case that $\mathcal{A}_j$ is of form (c7) here. Without loss of generality we assume that $x_1+x_2=x_1+x_3+x_4=x_1+x_5+x_6=x_2+x_3+x_5$. Assume to the contrary that $|\Sigma(S)| \leq 19$. Let

$b_1=x_1=x_3+x_5=x_4+x_6$,

$b_2=x_2=x_3+x_4=x_5+x_6$,

$b_3=x_3$,

$b_4=x_4$,

$b_5=x_5$,

$b_6=x_6$,

$b_7=x_1+x_2=x_1+x_3+x_4=x_1+x_5+x_6=x_2+x_3+x_5=x_2+x_4+x_6$,

$b_8=x_1+x_3=x_2+x_6=x_3+x_4+x_6$,

$b_9=x_1+x_4=x_2+x_5=x_3+x_4+x_5$,

$b_{10}=x_1+x_5=x_2+x_4=x_4+x_5+x_6$,

$b_{11}=x_1+x_6=x_2+x_3=x_3+x_5+x_6$,

$b_{12}=x_1+x_2+x_3=x_1+x_3+x_5+x_6=x_2+x_3+x_4+x_6$,

$b_{13}=x_1+x_2+x_4=x_1+x_4+x_5+x_6=x_2+x_3+x_4+x_5$,

$b_{14}=x_1+x_2+x_5=x_1+x_3+x_4+x_5=x_2+x_4+x_5+x_6$,

$b_{15}=x_1+x_2+x_6=x_1+x_3+x_4+x_6=x_2+x_3+x_5+x_6$,

$b_{16}=x_1+x_3+x_5=x_2+x_4+x_6=x_2+x_3+x_4=x_2+x_5+x_6=x_3+x_4+x_5+x_6$,

$b_{17}=x_1+x_2+x_3+x_4=x_1+x_2+x_5+x_6=x_1+x_3+x_4+x_5+x_6$,

$b_{18}=x_1+x_2+x_3+x_5=x_1+x_2+x_4+x_6=x_2+x_3+x_4+x_5+x_6$,

$b_{19}=x_1+x_2+x_3+x_4+x_5+x_6$.

A straightforward computation shows that $b_1, b_2, \ldots , b_{19}$ are pairwise distinct. Then
\begin{equation*}
B=\{b_i \mid 1\le i \le 19 \} = \Sigma(S),
\end{equation*}
and hence $b_{20}=x_3+x_6 \in \Sigma(S)$. Since $S$ is zero-sum free and $S$ contains no elements of order 2, we infer that $b_{20}\in \{b_4,b_5,b_{9},b_{10},b_{13},b_{14}\}$.

If $b_{20}\in \{b_4,b_5\}$, it is easy to verify that $b_{21}=x_1+x_2+x_4+x_5 \not\in B$, yielding a contradiction. If $b_{20}\in \{b_{9},b_{10}\}$, then $b_{22}=x_1+x_2+x_3+x_6 \not\in B$, yielding a contradiction. If $b_{20}\in \{b_{13},b_{14}\}$, then $b_{23}=x_4+x_5 \not\in B$,  yielding a contradiction too.
\qed

\bigskip
\begin{center}
{\bf Proof of Lemma \ref{19}}
\end{center}

Assume that $|\Sigma(S)|= 19$. We will show the lemma for the cases $\mathcal{A}_j \in  \{\mbox{b1, b4, c1, c2, ..., c6}\}.$

{\bf Proof of the case that $\mathcal{A}_j$ is of form (b1).} Without loss of generality we assume that $x_1=x_2+x_3+x_4+x_5=x_2+x_6=x_3+x_4+x_6$. Let

$b_1=x_1=x_2+x_3+x_4+x_5=x_2+x_6=x_3+x_4+x_6$,

$b_2=x_2=x_3+x_4$,

$b_3=x_3$,

$b_4=x_6=x_2+x_5=x_3+x_4+x_5$,

$b_5=x_1+x_2=x_1+x_3+x_4=x_2+x_3+x_4+x_6$,

$b_{6}=x_1+x_3=x_2+x_3+x_6$,

$b_7=x_1+x_5=x_2+x_5+x_6=x_3+x_4+x_5+x_6$,

$b_{8}=x_1+x_6=x_1+x_2+x_5=x_1+x_3+x_4+x_5=x_2+x_3+x_4+x_5+x_6$,

$b_9=x_3+x_6=x_2+x_3+x_5$,

$b_{10}=x_1+x_2+x_6=x_1+x_3+x_4+x_6=x_1+x_2+x_3+x_4+x_5$,

$b_{11}=x_1+x_3+x_5=x_2+x_3+x_5+x_6$,

$b_{12}=x_1+x_3+x_6=x_1+x_2+x_3+x_5$,

$b_{13}=x_1+x_2+x_5+x_6=x_1+x_3+x_4+x_5+x_6$,

$b_{14}=x_1+x_2+x_3+x_5+x_6$,

$b_{15}=x_1+x_2+x_3+x_4+x_5+x_6$.

Since $S$ is a zero-sum free subset and $S$ contains no elements of order~2, by Lemma~\ref{lemma T} we infer that $b_1,b_2,\ldots,b_{15}$ are pairwise distinct. Let $B=\{b_1,b_2,\ldots,b_{15}\}$. Note that $$b_{16}=x_5 \notin B\setminus \{b_5, b_{6}\}.$$ We distinguish three cases.

{\bf Case 1.} $b_{16}=b_5$. That is $x_5=x_1+x_2=x_2+x_3+x_4+x_6=x_1+x_3+x_4$. Let $b_{17}=x_3+x_5$, $b_{18}=x_3+x_5+x_6$, $b_{19}=x_1+x_3+x_5+x_6$, $b_{20}=x_2+x_3$, $b_{21}=x_2+x_3+x_4$, $b_{22}=x_1+x_5+x_6$, then
$(B\cup \{b_{17}, b_{18}, b_{19}, b_{20}, b_{21}, b_{22}\})\setminus\{b_{12}\}$ is a set of 20 distinct elements,
yielding a contradiction.

{\bf Case 2.} $b_{16}=b_6$. That is $x_5=x_1+x_3=x_2+x_3+x_6$. Then $B\cup\{b_{17}, b_{20}, b_{22}\}$ is a set of 18 distinct elements. Let $N=B\cup\{b_{17}, b_{20}, b_{22}\}$, then $b_{23}=x_1+x_4+x_5\notin N\setminus \{b_{17}\}$.

If $b_{23}=b_{17}$, let $b_{24}=x_2+x_4+x_5$, $b_{25}=x_1+x_4+x_5+x_6$. Then $N\cup\{b_{24}, b_{25}\}$ is a set of 20 distinct elements, yielding a contradiction.

If $b_{23}\neq b_{17}$, then $b_{26}=x_1+x_4\notin (N\cup\{b_{23}\})\setminus \{b_{11}, b_{12}\}$. If $b_{26}=b_{11}$, then $b_{27}=x_4+x_5\notin N\cup\{b_{23}\}$, yielding a contradiction. If $b_{26}=b_{12}$,
then $b_{25}\notin (N\cup\{b_{23}\})\setminus \{b_3\}$. If $b_{25}\neq b_3$, then $b_{25}\notin N\cup\{b_{23}\}$, yielding a contradiction. If $b_{25}=b_3$, then $x_1=6x_2$,
$x_3=-2x_2$, $x_4=3x_2$, $x_5=4x_2$, $x_6=5x_2$ and $\ord(x_2)=20$. Let $g=x_2$, then $S$ is of form (a3). If $b_{26}\neq b_{11}, b_{12}$, then $b_{26}\notin N\cup\{b_{23}\}$, yielding a contradiction.

{\bf Case 3.} $b_{16}\neq b_5, b_6$. Then $b_{16}\notin B$, $b_{17}\notin (B\cup\{b_{16}\})\setminus\{b_5\}$.

{\bf Subcase 3.1.} $b_{17}=b_5$, that is $x_3+x_5=x_1+x_2=x_2+x_3+x_4+x_6=x_1+x_3+x_4$, then $B\cup\{b_{16}, b_{19}, b_{21}\}$ is a set of 18 distinct elements. Let $E=B\cup\{b_{16}, b_{19}, b_{21}\}$, then $b_{30}=x_3+x_6\notin E\setminus\{b_3, b_6\}$.

If $b_{30}\neq b_3, b_6$, then $b_{30}\notin E$, $b_{31}=x_1+x_2+x_3\notin E$. If $b_{31}\neq b_{30}$, then $b_{31}\notin E\cup\{b_{30}\}$, yielding a contradiction. If $b_{31}=b_{30}$, then $b_{20}\notin E\cup\{b_{30}\}$, yielding a contradiction.

If $b_{30}=b_3$, then $b_{28}=x_2+x_4\notin E\setminus \{b_{14}\}$. If $b_{28}=b_{14}$, then $b_{27}=x_4+x_5\notin E$, $b_{29}=x_4+x_6\notin E\cup\{b_{27}\}$, yielding a contradiction. If $b_{28}\neq b_{14}$,
then $b_{29}\notin (E\cup\{b_{28}\})\setminus\{b_{14}\}$. If $b_{29}\neq b_{14}$, then $b_{29}\notin E\cup\{b_{28}\}$, yielding a contradiction. If $b_{29}=b_{14}$, then $S=x_2\cdot x_5\cdot (x_2+x_5)\cdot (x_2+2x_5)\cdot (2x_2+x_5)\cdot
(4x_2+4x_5)$ and $2x_2=2x_5$, $\ord(x_2)=\ord(x_5)=10$. Let $x_2=g, x_5=h$, then $S$ is of form (a5).

If $b_{30}=b_6$, then $B\cup\{b_{16}, b_{19}, b_{22}, b_{27}\}$ is a set of 19 distinct elements. Let $F=B\cup\{b_{16}, b_{19}, b_{22}, b_{27}\}$, then $b_{28}\notin F\setminus\{b_{14}\}$. If $b_{28}\neq b_{14}$, then $b_{28}\notin F$, yielding a contradiction. If $b_{28}=b_{14}$, then $b_{29}\notin F\setminus\{b_{3}\}$. If $b_{29}\neq b_{3}$, then $b_{29}\notin F$, yielding a contradiction. If $b_{29}=b_3$, then
$x_1=6x_2$, $x_3=3x_2$, $x_4=-2x_2$, $x_5=4x_2$, $x_6=5x_2$ and $\ord(x_2)=20$. Let $x_2=g$, then $S$ is of form (a3).

{\bf Subcase 3.2.} $b_{17}\neq b_5$, then $b_{17}\notin B\cup\{b_{16}\}$, $b_{19}\notin (B\cup\{b_{16}, b_{17}\})\setminus\{b_2, b_5\}$.

If $b_{19}\neq b_2, b_5$, then $B\cup\{b_{16}, b_{17}, b_{19}\}$ is a set of 18 distinct elements. Let $G=B\cup\{b_{16}, b_{17}, b_{19}\}$, then $b_{18}\notin G\setminus\{b_5\}$. If $b_{18}\neq b_5$, then $b_{22}\notin
(G\cup\{b_{18}\})\setminus\{b_2, b_6\}$. If $b_{22}\neq b_2, b_6$, then $b_{22}\notin G\cup\{b_{18}\}$, yielding a contradiction. If $b_{22}=b_2$, then $b_{30}\notin G\cup\{b_{18}\}$, yielding a contradiction. If $b_{22}=b_6$, then $b_{34}=x_1+x_2+x_3+x_6\notin G\cup\{b_{18}\}$, yielding a contradiction.
If $b_{18}=b_5$, then $b_{22}\notin G\setminus\{b_6\}$. If $b_{22}\neq b_6$, then $b_{22}\notin G$, $b_{26}\notin G\cup\{b_{22}\}$, yielding a contradiction. If $b_{22}=b_6$, then $G\cup\{b_{29}, b_{34}\}$ is a set of 20 distinct
elements, yielding a contradiction.

If $b_{19}=b_2$, then $B\cup\{b_{16}, b_{17}, b_{18}\}$ is a set of 18 distinct elements. Let $H=B\cup\{b_{16}, b_{17}, b_{18}\}$, then $b_{26}\notin H\setminus\{b_3\}$. If $b_{26}\neq b_3$, then $b_{26}\notin H$, $b_{22}\notin (H\cup\{b_{26}\})\setminus\{b_6\}$. If $b_{22}\neq b_6$, then $b_{22}\notin H\cup\{b_{26}\}$, yielding a contradiction. If $b_{22}=b_6$, then $b_{32}=x_1+x_4+x_6\notin H\cup\{b_{26}\}$, yielding a contradiction. If $b_{26}=b_3$, then $b_{29}\notin H$, $b_{22}\notin H\cup\{b_{29}\}$, yielding a contradiction.

If $b_{19}=b_5$, then $B\cup\{b_{16}, b_{17}, b_{26}\}$ is a set of 18 distinct elements. Let $I=B\cup\{b_{16}, b_{17}, b_{26}\}$, then $b_{30}\notin H\setminus\{b_6\}$. If $b_{30}\neq b_6$, then $b_{33}=x_1+x_2+x_4+x_6\notin (I\cup\{b_{30}\})\setminus\{b_3\}$. If $b_{33}\neq b_3$, then $b_{33}\notin I\cup\{b_{30}\}$, yielding a contradiction. If $b_{33}=b_3$, then $b_{20}\notin I\cup\{b_{30}\}$, yielding a contradiction. If $b_{30}=b_6$, then $b_{32}\notin I$, $b_{25}\notin I\cup\{b_{32}\}$, yielding a contradiction.

{\bf The prove of the case that $\mathcal{A}_j$ is of form (b4).} Without loss of generality we assume that $x_1=x_2+x_3+x_4=x_2+x_5+x_6=x_3+x_5$. Let

$c_1=x_1=x_2+x_3+x_4=x_2+x_5+x_6=x_3+x_5$,

$c_2=x_2$,

$c_3=x_3=x_2+x_6$,

$c_4=x_5=x_2+x_4$,

$c_5=x_1+x_2=x_2+x_3+x_5$,

$c_{6}=x_1+x_3=x_1+x_2+x_6=x_2+x_3+x_5+x_6$,

$c_7=x_1+x_4=x_3+x_4+x_5=x_2+x_4+x_5+x_6$,

$c_{8}=x_1+x_5=x_1+x_2+x_4=x_2+x_3+x_4+x_5$,

$c_9=x_1+x_6=x_3+x_5+x_6=x_2+x_3+x_4+x_6$,

$c_{10}=x_2+x_3$,

$c_{11}=x_2+x_5$,

$c_{12}=x_3+x_4=x_5+x_6=x_2+x_4+x_6$,

$c_{13}=x_1+x_3+x_4=x_1+x_5+x_6=x_1+x_2+x_4+x_6=x_2+x_3+x_4+x_5+x_6$,

$c_{14}=x_1+x_3+x_5=x_1+x_2+x_3+x_4=x_1+x_2+x_5+x_6$,

$c_{15}=x_1+x_2+x_3+x_4+x_5+x_6$.

Since $S$ is a zero-sum free subset and $S$ contains no elements of order 2, by Lemma~\ref{lemma T} we infer that $c_1,c_2,\ldots,c_{15}$ are pairwise distinct. Let $C=\{c_1,c_2,\ldots,c_{15}\}$. Note that $$c_{16}=x_1+x_2+x_3+x_5+x_6 \notin C\setminus \{c_7\}.$$ We distinguish two cases.

{\bf Case 1.} $c_{16}=c_7$. That is $x_1+x_2+x_3+x_5+x_6=x_1+x_4=x_3+x_4+x_5=x_2+x_4+x_5+x_6$. Then $c_{17}=x_2+x_3+x_6\notin C\setminus\{c_{11}\}$.

{\bf Subcase 1.1.} $c_{17}=c_{11}$. Let $c_{18}=x_1+x_4+x_6$, $c_{19}=x_4+x_5+x_6$,
$c_{20}=x_1+x_4+x_5$, $c_{21}=x_1+x_3+x_4+x_5$. Then $C\cup\{c_{18}, c_{19}, c_{20}, c_{21}\}$ is a set of 19 distinct elements. Let $J=C\cup\{c_{18}, c_{19}, c_{20}, c_{21}\}$. Then $c_{22}=x_1+x_4+x_5+x_6\notin J\setminus\{c_2\}$.
If $c_{22}\neq c_2$, then $c_{22}\notin J$, yielding a contradiction. If $c_{22}=c_2$, then $x_1=5x_3$, $x_2=-2x_3$,
$x_4=6x_3$, $x_5=4x_3$, $x_6=3x_3$ and $\ord(x_3)=20$. Let $g=x_3$, then $S$ is of form (a3).

{\bf Subcase 1.2.} $c_{17}\neq c_{11}$. Then $c_{17}\notin C$, $c_{23}=x_2+x_4+x_5\notin C$.

If $c_{23}\neq c_{17}$, then $c_{23}\notin C\cup\{c_{17}\}$, $c_{24}=x_1+x_2+x_3+x_4+x_5\notin C\cup\{c_{23}\}$.
If $c_{24}\neq c_{17}$, then $c_{24}\notin C\cup\{c_{17}, c_{23}\}$, $c_{19}\notin (C\cup\{c_{17}, c_{23}, c_{24}\})\setminus\{c_{11}\}$. If $c_{19}\neq c_{11}$, then $c_{19}\notin C\cup\{c_{17}, c_{23}, c_{24}\}$. Let
$K=C\cup\{c_{17}, c_{19}, c_{23}, c_{24}\}$, then $c_{20}\notin K\setminus\{c_2\}$. If $c_{20}\neq c_2$, then $c_{20}\notin K$, yielding a contradiction. If $c_{20}= c_2$, then $c_{22}\notin K$, yielding a contradiction.
If $c_{19}=c_{11}$, then $c_{25}=x_1+x_2+x_5\notin C\cup\{c_{17}, c_{23}, c_{24}\}$, $c_{20}\notin C\cup\{c_{17}, c_{23}, c_{24}, c_{25}\}$, yielding a contradiction. If $c_{24}= c_{17}$, then $C\cup\{c_{17}, c_{20}, c_{21}, c_{23}, c_{25}\}$ is a set of 20 distinct elements, yielding a contradiction.

If $c_{23}= c_{17}$, that is $x_2+x_4+x_5=x_2+x_3+x_6$. Let $c_{26}=x_4+x_6$, $c_{27}=x_3+x_4+x_6$. Then $C\cup\{c_{17}, c_{19}, c_{26}, c_{27}\}$ is a set of 19 distinct elements. Let $L=C\cup\{c_{17}, c_{19}, c_{26}, c_{27}\}$, then $c_{18}\notin L\setminus\{c_2\}$.
If $c_{18}\neq c_2$, then $c_{18}\notin L$, yielding a contradiction. If $c_{18}= c_2$, then $S=x_3\cdot x_5\cdot (x_3+x_5)\cdot (x_3+2x_5)\cdot (2x_3+x_5)\cdot (4x_3+4x_5)$ where  $2x_3=2x_5$ and $\ord(x_3)=\ord(x_5)=10$. Let $g=x_3$, $h=x_5$, then $S$ is of form (a5).

{\bf Case 2.} $c_{16}\neq c_7$. Then $c_{16}\notin C$, $c_{28}=x_6\notin (C\cup\{c_{6}\})\setminus\{c_5, c_8, c_{11}\}$.

{\bf Subcase 2.1.} $c_{28}=c_5$. That is $x_6=x_1+x_2=x_2+x_3+x_5$. Let $c_{29}=x_1+x_3+x_5+x_6$, $c_{30}=x_1+x_3+x_4+x_6$. Then $C\cup\{c_{16}, c_{18}, c_{21}, c_{29}, c_{30}\}$ is a set of 20 distinct elements,
yielding a contradiction.

{\bf Subcase 2.2.} $c_{28}= c_8$. Then $c_{31}=x_1+x_3+x_6\notin C\cup\{c_{16}\}$, $c_{29}\notin C\cup\{c_{16}, c_{31}\}$, $c_{17}\notin (C\cup\{c_{16}, c_{29}, c_{31}\})\setminus\{c_7\}$.

If $c_{17}\neq c_7$, then $c_{17}\notin C\cup\{c_{16}, c_{29}, c_{31}\}$. Let $P=C\cup\{c_{16}, c_{17}, c_{29}, c_{31}\}$, then $c_{32}=x_1+x_2+x_3\notin P\setminus\{c_7\}$. If $c_{32}\neq c_7$, then $c_{32}\notin P$, yielding a contradiction. If $c_{32}= c_7$, then $c_{23}\notin P$, yielding a contradiction.

If $c_{17}= c_7$, let $c_{33}=x_1+x_2+x_3+x_6$, then $c_{16}\notin C\cup\{c_{29}, c_{31}, c_{33}\}$, $c_{18}\notin C\cup\{c_{29}, c_{31}, c_{33}\}$. If $c_{16}\neq c_{18}$, then $C\cup\{c_{16}, c_{18}, c_{29}, c_{31}, c_{33}\}$ is a set of 20 distinct elements, yielding a contradiction. If $c_{16}= c_{18}$, let $Q=C\cup\{c_{16}, c_{29}, c_{31}, c_{33}\}$, then $c_{30}\notin Q\setminus\{c_2\}$. If $c_{30}\neq c_2$, then $c_{30}\notin Q$, yielding a contradiction.
If $c_{30}=c_2$, then $x_1=5x_5$, $x_2=-2x_5$, $x_3=4x_5$, $x_4=3x_5$, $x_6=6x_5$ and $\ord(x_5)=20$. Let $g=x_5$, then
$S$ is of form (a3).

{\bf Subcase 2.3.} $c_{28}= c_{11}$. Then $C\cup\{c_{16}, c_{29}, c_{30}\}$ is a set of 18 distinct elements. Let $R=C\cup\{c_{16}, c_{29}, c_{30}\}$, then $c_{31}\notin R\setminus\{c_7\}$.

If $c_{31}\neq c_7$, then $c_{31}\notin R$, $c_{32}\notin (R\cup\{c_{31}\})\setminus\{c_7\}$. If $c_{32}\neq c_7$,
then $c_{32}\notin R\cup\{c_{31}\}$, yielding a contradiction. If $c_{32}=c_7$,
then $c_{26}\notin R\cup\{c_{31}\}$, yielding a contradiction. If $c_{31}=c_7$, then $R\cup\{c_{21}, c_{22}\}$ is a set
of 20 distinct elements.

{\bf Subcase 2.4.} $c_{28}\neq c_5, c_8, c_{11}$. Then $C\cup\{c_{16}, c_{24}, c_{25}\}$ is a set of 18 distinct elements. Let $T=C\cup\{c_{16}, c_{24}, c_{25}\}$, then $c_{32}\notin T\setminus\{c_7, c_{12}\}$.

If $c_{32}= c_7$, then $T\cup\{c_{20}, c_{21}\}$ is a set of 20 distinct elements, yielding a contradiction. If $c_{32}= c_{12}$, then $T\cup\{c_{21}, c_{22}\}$ is a set of 20 distinct elements, yielding a contradiction. If $c_{32}\neq c_7, c_{12}$,
then $c_{32}\notin T$, $c_{34}=x_1+x_2+x_3+x_5\notin (T\cup\{c_{32}\})\setminus\{c_{12}\}$. If $c_{34}\neq c_{12}$,
then $c_{34}\notin T\cup\{c_{32}\}$, yielding a contradiction. If $c_{34}= c_{12}$,
then $c_{18}\notin T\cup\{c_{32}\}$, yielding a contradiction too.

{\bf The proof of the case that $\mathcal{A}_j$ is of form (c1).} Without loss of generality we assume that $x_1+x_2=x_3+x_4=x_1+x_3+x_5+x_6=x_2+x_4+x_5+x_6=x_1+x_4+x_5$. Let

$d_1=x_1=x_2+x_6=x_4+x_5+x_6$,

$d_2=x_2=x_4+x_5=x_3+x_5+x_6$,

$d_3=x_3=x_1+x_5=x_2+x_5+x_6$,

$d_4=x_4=x_3+x_6=x_1+x_5+x_6$,

$d_5=x_1+x_2=x_3+x_4=x_1+x_3+x_5+x_6=x_2+x_4+x_5+x_6=x_1+x_4+x_5$,

$d_{6}=x_1+x_3=x_2+x_4=x_1+x_2+x_5+x_6=x_3+x_4+x_5+x_6=x_2+x_3+x_6$,

$d_7=x_1+x_4=x_1+x_3+x_6=x_2+x_4+x_6$,

$d_{8}=x_2+x_3=x_1+x_2+x_5=x_3+x_4+x_5$,

$d_9=x_5+x_6$,

$d_{10}=x_1+x_2+x_3=x_1+x_3+x_4+x_5=x_2+x_3+x_4+x_5+x_6$,

$d_{11}=x_1+x_2+x_4=x_1+x_2+x_3+x_6=x_1+x_3+x_4+x_5+x_6$,

$d_{12}=x_1+x_2+x_6=x_3+x_4+x_6=x_1+x_4+x_5+x_6$,

$d_{13}=x_1+x_3+x_4=x_2+x_3+x_4+x_6=x_1+x_2+x_4+x_5+x_6$,

$d_{14}=x_1+x_3+x_5=x_2+x_4+x_5=x_2+x_3+x_5+x_6$,

$d_{15}=x_2+x_3+x_4=x_1+x_2+x_4+x_5=x_1+x_2+x_3+x_5+x_6$,

$d_{16}=x_1+x_2+x_3+x_4$,

$d_{17}=x_1+x_2+x_3+x_4+x_5+x_6$.

Since $S$ is a zero-sum free subset and $S$ contains no elements of order 2, by Lemma~\ref{lemma T} we infer that $d_1,d_2,\ldots,d_{17}$ are pairwise distinct. Let $D=\{d_1,d_2,\ldots,d_{17}\}$. Note that $$d_{18}=x_5 \notin D\setminus \{d_7, d_{16}\}.$$ We distinguish three cases.

{\bf Case 1.} $d_{18}=d_7$. That is $x_5=x_1+x_4=x_1+x_3+x_6=x_2+x_4+x_6$. Then $d_{19}=x_1+x_6\notin D$, $d_{20}=x_4+x_6\notin D\cup\{d_{19}\}$. Let $X=D\cup\{d_{19}, d_{20}\}$, then $d_{21}=x_2+x_3+x_4+x_5\notin X\setminus\{d_{20}\}$. If $d_{21}\neq d_{20}$, then $d_{21}\notin X$, yielding a contradiction. If $d_{21}=d_{20}$,
then $S=x_1\cdot x_4\cdot (x_1+x_4)\cdot (x_1+2x_4)\cdot (2x_1+x_4)\cdot (4x_1+4x_4)$ where $2x_1=2x_4$ and $\ord(x_1)=\ord(x_4)=10$. Let $g=x_1$, $h=x_4$, then $S$ is of form (a5).

{\bf Case 2.} $d_{18}=d_{16}$. That is $x_5=x_1+x_2+x_3+x_4$. Then $d_{22}=x_2+x_5\notin D$, $d_{23}=x_3+x_5\notin D\cup\{d_{22}\}$. Let $Y=D\cup\{d_{22}, d_{23}\}$, then $d_{24}=x_1+x_3+x_4+x_6\notin Y\setminus\{d_{23}\}$. If $d_{24}\neq d_{23}$, then $d_{24}\notin Y$, yielding a contradiction. If $d_{24}=d_{23}$,
then $S=x_2\cdot x_3\cdot (x_2+x_3)\cdot (x_2+2x_3)\cdot (2x_2+x_3)\cdot (4x_2+4x_3)$ where $2x_2=2x_3$ and  $\ord(x_2)=\ord(x_3)=10$. Let $g=x_2$, $h=x_3$, then $S$ is of form (a5).

{\bf Case 3.} $d_{18}\neq d_7, d_{16}$. Then $d_{18}\notin D$, $d_{25}=x_1+x_2+x_4+x_6\notin (D\cup\{d_{18}\})\setminus\{d_8\}$. If $d_{25}=d_8$, then $d_{19}\notin D\cup\{d_{18}\}$, $d_{26}=x_1+x_2+x_3+x_4+x_6\notin D\cup\{d_{18}, d_{19}\}$, yielding a contradiction. If $d_{25}\neq d_8$, then
$d_{26}\notin D\cup\{d_{25}\}$, $d_{24}\notin (D\cup\{d_{25}, d_{26}\})\setminus\{d_8\}$. If $d_{24}\neq d_8$, then
$d_{24}\notin D\cup\{d_{25}, d_{26}\}$, yielding a contradiction. If $d_{24}= d_8$, then
$d_{18}\notin D\cup\{d_{25}, d_{26}\}$, yielding a contradiction too.

If $\mathcal{A}_j$ has one of the forms (c2), (c3), (c4), (c5) and(c6). Then we have one of the following holds correspondingly:
\begin{equation*}
x_{\tau(2)}=x_{\tau(3)}+x_{\tau(5)}+x_{\tau(6)}=x_{\tau(4)}+x_{\tau(5)}=x_{\tau(1)}+x_{\tau(3)},
\end{equation*}
\begin{equation*}
x_{\tau(3)}=x_{\tau(4)}+x_{\tau(5)}+x_{\tau(6)}=x_{\tau(1)}+x_{\tau(5)}=x_{\tau(2)}+x_{\tau(6)},
\end{equation*}
\begin{equation*}
x_{\tau(1)}=x_{\tau(2)}+x_{\tau(5)}+x_{\tau(6)}=x_{\tau(3)}+x_{\tau(5)}=x_{\tau(4)}+x_{\tau(6)},
\end{equation*}
\begin{equation*}
x_{\tau(2)}=x_{\tau(3)}+x_{\tau(5)}+x_{\tau(6)}=x_{\tau(1)}+x_{\tau(5)}=x_{\tau(3)}+x_{\tau(4)},
\end{equation*}
\begin{equation*}
x_{\tau(2)}=x_{\tau(3)}+x_{\tau(5)}+x_{\tau(6)}=x_{\tau(1)}+x_{\tau(5)}=x_{\tau(3)}+x_{\tau(4)}.
\end{equation*}
It reduces to the case that $\mathcal{A}_j$  of form (b3). This completes the proof.
\qed

\end{document}